\documentclass[a4paper,12pt,oneside]{amsart}
\usepackage{amsfonts}
\usepackage{amsmath}
\usepackage{amsthm}
\usepackage{fancyhdr}

\newtheorem{theorem}[equation]{Theorem}

\newtheorem{proposition}[equation]{Proposition}

\newcommand{\R}{\mathbf{R}}
\newcommand{\C}{\mathbf{C}}

\newcommand{\N}{\mathbf{N}}  
\newcommand{\Z}{\mathbf{Z}}

\newcommand{\Ge}{\mathfrak{g}}

\newcommand{\Ges}{\mathfrak{g}^{\star}}

\newcommand{\q}{\mathbf{q}}

\numberwithin{equation}{section}
\newcommand{\ignore}[1]{{}}
\usepackage{amsmath}

\newcommand{\Schw}{\mathcal{S}}

\renewcommand{\phi}{\varphi}
\newcommand{\op}{{\rm Op\,}}
\newcommand{\Ltwo}{{\mathbf{ L^2}}}

\begin{document}
 \baselineskip=17pt 
\thispagestyle{empty}
\small
{\noindent\footnotesize
   Pawe{\l} G{\l}owacki,\\
   Institute of Mathematics,  University of Wroc{\l}aw,\\
   pl. Grunwaldzki 2/4,    50-384 Wroc{\l}aw, Poland,\\
   {\tt glowacki@math.uni.wroc.pl}}

\vspace{8ex}
\begin{center}\bf
Composition and $\Ltwo$-boundedness
of flag kernels
\end{center}

\vspace{4ex}
{\small
\vspace{2ex}
}

%\ignore{
\vspace{4ex}
{\small
\vspace{1ex}
{\bf Abstract.} We prove the composition and  $L^2$-boundedness theorems for the Nagel-Ricci-Stein flag kernels related to the natural gradation of  homogeneous groups. }

\vspace{4ex}
In \cite{nagel}, Nagel, Ricci, and Stein introduce a notion of a flag kernel which generalizes that of a  singular integral kernel of Calder\'on-Zygmund as a tool in their investigation of  operators naturally associated with the $\bar{\partial}_b$ on some $CR$ submanifolds of $\C^n\times\C^n$. \textit{A flag kernel} $K$ on a Euclidean vector space $V$ endowed with a family of dilations and a corresponding homogeneous norm $x\to|x|$ is a tempered distribution associated with  gradations
$$
V=\bigoplus_{j=1}^RV_j,
\qquad
V^{\star}=\bigoplus_{j=1}^RV_j^{\star}
$$
of the space and its dual. The Fourier transform of $K$ is required to be smooth for $\xi_R\neq0$ and satisfy
\begin{equation}\label{flag}
|D^{\alpha}\widehat{K}(\xi)|\le C_{\alpha}|\xi|_1^{-|\alpha_1|}|\xi|_2^{-|\alpha_2|}\dots|\xi|_R^{-|\alpha_R|},
\end{equation}
where
$$
|\xi|_j=\sum_{k=j}^R|\xi_k|,
\qquad
\xi=\sum_{k=1}^R\xi_k \in V^{\star},
$$
and $\alpha_j$ are submultiindices corresponding to the spaces $V_j^{\star}$. Actually, the authors define the flag kernels directly in terms of the smoothness and cancellation properties of the kernels, and then prove that the multiplier condition (\ref{flag}) is an equivalent possibility of definition.

They prove that if $V$ is the Lie algebra of the homogeneous group identified with the group itself, dilations are automorphisms of the group,  spaces $V_j$ are homogeneous, and $[V_j,V_k]=\{0\}$ for $j\neq k$, then the composition of flag kernels associated with the same gradation is still a flag kernel. Moreover, under the same hypotheses any flag kernel $K$ defines a bounded operator 
$$
Kf(x)=f\star\widetilde{K}(x)=\int_Vf(xy)K(y)\,dy
$$
on $L^p(V)$ for $1<p<\infty$.

A natural question arises, whether the composition and boundedness properties still hold if the underlying gradation is the natural one of a homogeneous group. Note that for homegeneous groups of step bigger than 2 the commutator condition is no longer  satisfied. We provide an answer in the affirmative under somewhat relaxed assumptions on the kernels and $p=2$.

The results presented here depend heavily on the symbolic calculus of \cite{glowacki} and can be regarded as an example of usefulness of such a calculus. There occurs a striking resemblance between the estimates defining flag kernels and those of the calculus which has been created and developed quite independently.

The problem of the $L^p$-boundedness of flag kernels on arbitrary homogeneous groups will be dealt with in another paper \cite{glowacki1}.

I wish to thank Fulvio Ricci for a fruitful conversation concerning the subject of this paper.

Even though I had no opportunity to discuss the subject matter of this paper with Andrzej Hulanicki, it unavoidably bears signs of his  influence which can be traced back throughout  the whole of my mathematical work.

\section{Background}

Let $\Ge$ be a nilpotent Lie algebra with a fixed Euclidean structure and $\Ges$ its dual.  Let $\{\delta_t\}_{t>0}$, be a family of group dilations on
$\Ge$ and let
$$
\Ge_j=\{x\in\Ge: \delta_tx=t^{d_j}x\},
\hspace{2em}
1\le j\le R,
$$
where $1=d_1< d_2<\dots <d_R$. Then
\begin{equation}\label{grad}
\Ge=\bigoplus_{j=1}^R\Ge_j,
\qquad
\Ges=\bigoplus_{j=1}^R\Ges_j,
\end{equation}
and
$$
[\Ge_i,\Ge_j]\subset \left\{
\begin{array}{ll}
\Ge_k, & $ if $ d_i+d_j=d_k,\cr \{0\}, & $ if $ d_i+d_j\notin{\mathcal{D},}
\end{array}
\right.
$$
where $\mathcal{D}=\{d_j:1\le j\le R\}$. Let 
$$
\xi\to|\xi|=\sum_{j=1}^R\|\xi_j\|^{1/d_j}=\sum_{j=1}^R|\xi_j|
$$ 
be a homogeneous norm on $\Ges$..  We say that $\Ge$ is \textit{homogeneous of step} $R$. 

We shall also regard $\Ge$ as a Lie group with the
Campbell-Hausdorff multiplication
$$
x_1x_2=x_1+x_2+r(x_1,x_2),
$$
where
\begin{equation*}
\begin{split}
r(x_1,x_2)&=\frac{1}{2}[x_1,x_2]+\frac{1}{12}([x_1,[x_1,x_2]]+[x_2,[x_2,x_1]])
\cr
&+{1\over24}\left[x_2,\left[x_1,[x_2,x_1]\right]\right]+\dots
\end{split}
\end{equation*}
is the (finite) sum of terms of order at least $2$ in the Campbell-Hausdorff
series for $\Ge$. 

Let
$$
|\xi|_j=\sum_{k=j}^R|\xi_k|, 
\qquad
1\le j\le R,
$$
and let $|\xi|_{R+1}=0$. 
Let
$$
\q_\xi(\eta)=\sum_{j=1}^R\frac{\|\eta_j\|}{1+|\xi|_{j+1}},
\qquad
\xi,\eta\in\Ges,
$$
be a family of norms (a H\"ormander metric) on $\Ges$.  Let $g_j$ be a family of functions on $\Ges$ satisfying
$$
|\xi|_{j+1}\le g_j(\xi)\le|\xi|,
\qquad
1\le j\le R,
$$
and
$$
\left(\frac{1+g_j(\xi)}{1+g_j(\eta)}\right)^{\pm1}\le C(1+\q_\xi(\xi-\eta))^M
$$
for some $C>0$ and $M>0$. The metric $\q$ is fixed throughout the paper (cf \cite{glowacki}).

The class $S^m(\Ge)$, where $m\in\R$, is defined as the space of all $A\in\Schw'(\Ge)$ whose Fourier transforms are smooth and satisfy
$$
|D^{\alpha}\widehat{A}(\xi)|\le C_{\alpha}(1+|\xi|)^m\Pi_{j=1}^R(1+g_j(\xi))^{-|\alpha_j|},
\qquad
\xi\in\Ges,
$$
where $\alpha=(\alpha_1,\dots,\alpha_R)$ is a multiindex of length equal to the dimension of $\Ges$, and $\alpha_j$ are submultiindices corresponding to the subspaces $\Ges_j$. 
Note that the elements of $S^m(\Ge)$ have no singularity at infinity.

The space $S^m(\Ge)$ is a Fr\'echet space if equipped with the seminorms
$$
\|A\|_{\alpha}=\sup_{\xi\in\Ges}\Pi_{k=1}^R(1+g_k(\xi))^{|\alpha|}|D^{\alpha}\widehat{A}(\xi)|.
$$

The class $S^0(\Ge)$ is known to be an subalgebra of ${\mathcal{B}}(L^2(\Ge))$. More precisely, we have the following two propositions proved in \cite{glowacki}.

\begin{proposition}\label{calculus}
The mapping
$$
S^{m_1}(\Ge)\times S^{m_2}(\Ge)\ni(A,B)\mapsto A\star B\in S^{m_1+m_2}(\Ge)
$$
is  continuous.
\end{proposition}

\begin{proposition}\label{cv}
If $A\in S^0(\Ge)$, then 
$$
{\rm Op}(A)f(x)=\int_{\Ge}f(xy)A(dy),
\qquad
f\in\Schw(\Ge),
$$ 
extends to a  bounded operator on $L^2(\Ge)$, and the mapping
$$
S^0(\Ge)\ni A\mapsto{\rm Op}(A)\in\mathcal{B}(L^2(\Ge))
$$
is continuous.
\end{proposition}

\ignore{
Let $A_j\in S^{m_j}(\Ge)$, where $m_j\searrow-\infty$. Then there exists a distribution $A\in S^{m_1}$ such that
$$
A-\sum_{j=1}^NA_j\in S^{m_{N+1}}(\Ge)
$$
for every $N\in\N$. The distribution $A$ is unique modulo the Schwartz class $\Schw(\Ge)$.  We shall write
$$
A\approx\sum_{j=1}^{\infty}A_j
$$ 
and call  $A$ the asymptotic sum of the corresponding series (cf., e.g., H\"ormander \cite{hormander}, Proposition 18.1.3).
}

\section{Main results}

We extend the definition of a flag kernel of Nagel-Ricci-Stein to include all $K\in\Schw'(\Ge)$ whose Fourier transforms are smooth for $\xi_R\neq0$ and satisfy
\begin{equation}\label{multiplier}
|D^{\alpha}\widehat{K}(\xi)|\le C_{\alpha}\Pi_{j=1}^Rg_j(\xi)^{-|\alpha_j|},
\qquad
\xi_R\neq0,
\end{equation}
where the weight functions defined above are now additionally assumed to be homogeneous. Note that for $g_j(\xi)=|\xi|_j$ we get the usual flag kernels. Another interesting choice is $g_j(\xi)=|\xi|_{j+1}$. In the latter case the estimates of the derivatives in the direction of $\xi_R$ are irrelevant. Observe that if 
$$
<K_t,f>=\int_{\Ge}f(tx)K(x)\,dx,
$$
then the flag kernels $K_t$ satisfy the estimates (\ref{multiplier}) uniformly in $t>0$.

 We shall  need two cut-off functions. Let $\phi\in C^{\infty}(\Ges_R)$ be equal to $1$ for $1\le|\xi_R|\le2$ and vanish for $|\xi_R|\ge4$ and $|\xi_R|\le1/2$.  Let $\psi\in C^{\infty}(\Ges_R)$ be equal to $1$ for $1/2\le|\xi_R|\le4$ and vanish for $|\xi_R|\ge8$ and $|\xi_R|\le1/4$. Thus, in particular, $\phi\cdot\psi=\phi$.
\begin{theorem}
A composition of  flag kernels is also a flag kernel.
\end{theorem}

\begin{proof}
 Let $K=K_1\star K_2$, where $K_j$ are flag kernels. Then
$$
\widehat{K}(\xi)=\widehat{A_1\star A}_2(\xi),
\qquad
1\le|\xi_R|\le2,
$$
where
$$
\widehat{A}_j(\xi)=\widehat{K}_j(\xi)\phi(\xi_R),
$$
and $A_j\in S^0(\Ge)$. Therefore, by Proposition \ref{calculus},
\begin{equation}\label{estimate}
|D^{\alpha}\widehat{K}(\xi)|\le C_{\alpha}\Pi_{j=1}^Rg_j(\xi)^{-|\alpha_j|},
\qquad
1\le|\xi_R|\le2.
\end{equation}
Since
\begin{equation}\label{homo}
\widehat{(K_j)_t}(\xi)=\widehat{K}_j(t\xi),
\qquad
t>0,
\end{equation}
satisfy uniformly (\ref{multiplier}), we get the estimate (\ref{estimate}) for $1/t\le|\xi_R|\le2/t$, $t>0$, that is, for all $\xi_R\neq0$.
\end{proof}

\begin{theorem}
Let $K$ be a flag kernel. The operator $f\to f\star\widetilde{K}$ defined initially on $\Schw(\Ge)$ extends uniquely to a bounded operator on $L^2(\Ge)$.
\end{theorem}
\begin{proof}
For $f\in\Schw(\Ge)$ let 
$$
\widehat{f_n}(\xi)=\widehat{f}(\xi)\phi(2^{-n}\xi_R),
\qquad
n\in\Z.
$$
Then
$$
m\|f\|_2\le \sum_{n=-\infty}^{\infty}\|f_n\|_2\le M\|f\|_2,
\qquad
f\in\Schw(\Ge),
$$
for some $m,M>0$. Let $K_n$ be defined by 
$$
\widehat{K_n}(\xi)=\widehat{K}(\xi)\psi(2^{-n}\xi_R).
$$ 
Then the flag kernels $L_n=(K_n)_{2^n}$ are uniformly in $S^0(\Ge)$, and
\begin{align*}
\|\op(K)&f\|_2\le\frac{1}{m}\sum_{n\in\Z}\|\op(K_n)f_n\|_2\\
&=\frac{1}{m}\sum_{n\in\Z}2^{nQ/2}\|\op(L_n)(f_n)_{2^n}\|_2
\le\frac{C}{m}\sum_{n\in\Z}2^{nQ/2}\|(f_n)_{2^n}\|_2\\
&\le\frac{C}{m}\sum_{n\in\Z}\|f_n\|_2
\le\frac{CM}{m}\|f\|_2
\end{align*}
for a $C>0$, which completes the proof.
\end{proof}

 \end{document}